\documentclass[12pt,a4paper]{amsart}

\usepackage{fourier}
\usepackage{ulem}
\usepackage{amsthm}
\usepackage{amssymb}
\usepackage{mathrsfs}
\usepackage{amsmath}
\usepackage[hidelinks]{hyperref}
\hypersetup{
  pdftitle   = {Grothendieck groups of categories of abelian varieties},
  pdfauthor  = {Shnidman},
  pdfcreator = {\LaTeX\ with package \flqq hyperref\frqq}
}

\DeclareMathAlphabet{\mathpzc}{OT1}{pzc}{m}{it}

\newtheorem{theorem}{Theorem}[section]
\newtheorem*{theorem*}{Theorem}
\newtheorem{proposition}[theorem]{Proposition}
\newtheorem{lemma}[theorem]{Lemma}
\newtheorem*{lemma*}{Lemma}
\newtheorem{corollary}[theorem]{Corollary}
\newtheorem*{corollary*}{Corollary}

\newtheorem*{conjecture*}{Conjecture}

\theoremstyle{definition}
\newtheorem{definition}[theorem]{Definition}
\newtheorem{example}[theorem]{Example}

\theoremstyle{remark}
\newtheorem{remark}[theorem]{Remark}

\def\F{{\rm \mathbb{F}}}
\def\Z{{\rm \mathbb{Z}}}
\def\Q{{\rm \mathbb{Q}}}

\def\CC{{\rm \mathcal{C}}}
\def\A{{\rm \mathcal{A}}}
\def\O{{\rm \mathcal{O}}}

\def\a{{\rm \mathfrak{a}}}

\def\l{{\rm \mathfrak{l}}}

\def\ZZ{{\rm \Z_{+}}}
\def\QQ{{\rm \Q_{+}}}
\def\qq{{\rm \Q_{+}^2}}
\def\Hom{{\rm Hom}}
\def\Nrd{{\rm Nrd}}
\def\End{{\rm End}}
\def\char{{\rm char }}
\def\Mat{{\rm Mat}}
\def\GL{{\rm GL}}
\def\Ob{{\rm Ob}}
\def\Pic{{\rm Pic}}
\def\Nm{{\rm Nm}}
\def\dist{{\rm dist}}
\def\tot{{\rm tot}}
\def\ldeg{{\rm deg_\ell}}

\numberwithin{equation}{section}

\begin{document}

\title[Grothendieck groups of categories of abelian varieties]{Grothendieck groups of categories of abelian varieties}

\author{Ari Shnidman}
\address{Department of Mathematics, Boston College, Chestnut Hill, MA 02467}
\email{shnidman@bc.edu}
\subjclass[2010]{11G10, 14K99, 19A99}
\keywords{Abelian varieties, Grothendieck groups, elliptic curves}

\begin{abstract}
We compute the Grothendieck group of the category of abelian varieties over an algebraically closed field $k$.  We also compute the Grothendieck group of the category of $A$-isotypic abelian varieties, for any simple abelian variety $A$, assuming $k$ has characteristic 0, and for any elliptic curve $A$ in any characteristic.         
\end{abstract}

\maketitle

\section{Introduction}
The purpose of this note is to determine the Grothendieck groups of various categories of abelian varieties.    If $\CC$ is an exact category, then the Grothendieck group $K_0(\CC)$ is the quotient of the free abelian group generated by isomorphism classes in $\CC$ modulo the relations $[X] - [Y] + [Z]$, for any exact sequence $0 \to X \to Y \to Z \to 0$ in $\CC$.

Let $\A$ be the category of abelian varieties over an algebraically closed field $k$.  The morphisms in $\A$ are homomorphisms of abelian varieties.  Kernels do not necessarily exist in $\A$, but cokernels do exist, and $\A$ is an exact category.  

To compute the Grothendieck group $K_0(\A)$, it is helpful to consider the simpler category $\tilde \A$ of abelian varieties up to isogeny.  This category has the same objects as $\A$, and  
\[\Hom_{\tilde\A}(A,B) = \Hom(A,B) \otimes_\Z \Q,\] for $A,B \in \Ob(\tilde\A)$.  The category $\tilde\A$ is semisimple by Poincar\'e's reducibiliity theorem, so that 
\[K_0(\tilde\A) \cong \bigoplus_A \Z [A], \] where the direct sum is over representatives of simple isogeny classes of abelian varieties.  We have a surjective map $\iota \colon K_0(\A) \to K_0(\tilde \A)$ which sends an abelian variety to its isogeny class.  In fact: 
\begin{theorem}\label{entirecat}
The map $\iota$ is an isomorphism.  In particular, any additive function $\Ob(\A) \to G$ to an abelian group $G$ is an isogeny invariant.          
\end{theorem}

The proof of Theorem \ref{entirecat}, which was suggested to us by Julian Rosen, uses a trick involving non-isotypic abelian varieties to reduce to showing the following fact: for every finite group scheme $G$ over $k$ of prime order, there is an abelian variety $A$ over $k$ with an endomorphism whose kernel is isomorphic to $G$.  This fact is easily proved using the arithmetic of elliptic curves.  Thus, the proof of Theorem \ref{entirecat} exploits the fact that certain elliptic curves have extra endomorphisms, and does not shed much light on the structure of the category $\A$.     

For example, one consequence of Theorem \ref{entirecat} is that an abelian variety $A$ and its dual $\hat A$ determine the same class in $K_0(\A)$.  It is then natural to ask whether we can witness the relation $[A] = [\hat A]$ in $K_0(\A)$, using only short exact sequences that are intrinsic to $A$.  In other words, we ask for short exact sequences involving only abelian varieties that can be constructed from $A$ or $\hat A$ in some natural way, and that don't involve auxiliary abelian varieties such as CM elliptic curves.  One way to formalize this question is as follows.        

Let $A$ be a simple abelian variety of dimension $g$.  Write $\CC_A$ for the category of abelian varieties $B$ isogenous to $A^n$ for some $n \geq 0$, with morphisms as usual.  We will write $K_0(A)$ for $K_0(\CC_A)$, but notice that this group depends only on the isogeny class of $A$.  One can then ask whether it is true that $[A] = [\hat A]$ in $K_0(A)$.  It will follow from our main result below that the answer is typically no.  

Write $G(A)$ for the kernel of the map $\dim \colon K_0(A) \to \Z$ sending $[A]$ to $\dim A$.  The group $G(A)$ measures the difference between the category $\CC_A$ and its isogeny category $\tilde\CC_A$.  
We will describe $G(A)$ in terms of the endomorphism algebra $D = \End(A) \otimes \Q$, assuming $k$ has characteristic 0. So let us assume $\char \, k = 0$ until further notice.  Let $F$ be the center of $D$, and write $e = [F: \Q]$ and $d^2 = [D: F]$.  Also let $F^+$ be the set of non-zero elements of $F$ which are positive at all the ramified real places of $D$.
Then our main result is:     

\begin{theorem}\label{intro}
Let $\Q_+$ be the group of positive rational numbers under multiplication.  Then there is a canonical isomorphism 
\[\deg \colon G(A) \simeq \QQ\big/\Nm_{F/\Q}\left(F^+\right)^{2g/de},\]
which, for any two simple $A_1, A_2 \in \Ob(\CC_A)$, sends the class of $[A_1] - [A_2]$ to the class of the degree of any isogeny $A_1 \to A_2$.     
\end{theorem}

This shows that even though $\iota \colon K_0(\A) \to K_0(\tilde \A)$ is an isomorphism, there is still a big difference between the $A$-isotypic part of $\A$ and the $A$-isotypic part of $\tilde\A$.  Indeed:    

\begin{corollary}\label{torsion}
The group $G(A)$ is an infinite torsion group.
\end{corollary}
\begin{proof}
The group $\QQ\big/\Nm_{F/\Q}\left(F^+\right)^{2g/de}$ is torsion since every integer has a power which is a norm from $F$.  If $\QQ\big/\Nm_{F/\Q}\left(F^+\right)^{2g/de}$ is finite, then we must have $2g = de$.  By the classification of endomorphism algebras of abelian varieties in characteristic 0, this implies that $F$ is a CM field.  But then $\Nm_{F/\Q}(F^\times)$ has infinite index in $\Q^\times$, which contradicts the assumption that $\QQ\big/\Nm_{F/\Q}\left(F^+\right)$ is finite.       
\end{proof}

Duality fits nicely into this picture as well:  

\begin{theorem}\label{duality}
The automorphism $[B] \mapsto [\hat B]$ of $K_0(A)$ is inversion on $G(A)$.     
\end{theorem}

\begin{remark}
Note that duality does not induce inversion on all of $K_0(A)$, as $[A] \mapsto -[A]$ does not preserve dimension.
\end{remark}

This lets us exhibit many cases where $[A] \neq [\hat A]$ in $K_0(A)$.  

\begin{example}\label{surf}
If $A$ is an abelian surface with $\End(A) = \Z$, then Theorem \ref{intro} gives
\[G(A) \simeq \QQ/\Q_+^4 \simeq \bigoplus_\ell \Z/4\Z,\]
with the direct sum over all primes $\ell$.    
Suppose $A$ admits an isogeny $A_0 \to A$ of degree $\ell$ from a principally polarized surface $A_0$; in particular $A_0 \simeq \hat A_0$.  Then $[A] = [\hat A]$ in $K_0(A)$ if and only if $[A_0] - [A] = [A_0] - [\hat A]$ in $G(A)$.  By Theorem \ref{duality}, this would mean that the class of $[A_0] - [A]$ in $G(A)$ is its own inverse.  But the inverse of $\deg([A_0] - [A]) = \ell$ in $\QQ\big/\Q_+^4$ is $\ell^3$ and not $\ell$, so we must have $[A] \neq [\hat A]$ in $K_0(A)$.     
\end{example}

On the other hand, Theorem \ref{duality} immediately yields the following positive result, giving a canonical class in degree two in $K_0(A)$: 
\begin{theorem}\label{canonical}
If $A_1, A_2 \in \Ob(\CC_A)$ are simple, then $[A_1] + [\hat A_1]  = [A_2] + [\hat A_2]$ in $K_0(A)$.  In other words, the class $[A] + [\hat A]$ is independent of the choice of $A$ in its isogeny class.  
\end{theorem}

Theorem \ref{intro} and Corollary \ref{torsion} can fail quite dramatically in characteristic $p$.  The next two results show that $G(A)$ need not be infinite nor torsion in positive characteristic.

\begin{theorem}\label{supersing}
Suppose $\char \, k  = p > 0$ and let $A$ be a supersingular elliptic curve.  Then $G(A) = 0$.  In particular, $K_0(A) \simeq \Z$.     
\end{theorem}

\begin{theorem}\label{frob}
Suppose $\char \, k = p > 0$ and let $A$ be an elliptic curve with $\End(A) = \Z$.  Then $G(A)$ contains an element of infinite order.  
\end{theorem}

For general $A$ in characteristic $p$, determining the structure of $G(A)$ is somewhat subtle, and we hope to return to this question in future work.  It would also be interesting to compute the groups $G(A)$ when $k$ is a finite field.  The answer should be related to Milne's computation \cite{milne} of the size of the Ext group of two abelian varieties over a finite field.  When $A$ is an elliptic curve over a finite field, one can presumably deduce the answer from the results of the recent paper \cite{jordan}.            

The plan for the rest of the note is as follows.  In Section 2, we prove Theorem \ref{entirecat}.  In Section 3, we prove a useful criterion for two isogenous abelian varieties $B, B' \in \Ob(\CC_A)$ to determine the same class in $K_0(A)$.  In Section 4, we define the map $\deg$ in Theorem \ref{intro}.  In Section 5, we prove that $\deg$ is an isomorphism in characteristic 0.  In Section 6, we determine $G(E)$ for any elliptic curve $E$, in any characteristic. 

\section{The Grothendieck group $K_0(\A)$}
We let $k$ be an algebraically closed field and $\A$ the category of abelian varieties over $k$.   
\begin{theorem}\label{entirecat2}
There is an isomorphism $K_0(\A) \cong \bigoplus_A \Z[A]$, where the sum is over representatives $A$ of simple isogeny classes of abelian varieties.   
\end{theorem}

\begin{proof}
By Poincar\'e reducibility, it is enough to show that if $\phi\colon  A \to A'$ is an isogeny of abelian varieties, then $[A] = [A']$ in $K_0(\A)$.  We immediately reduce to the case where $\phi$ is an $\ell$-isogeny for some prime $\ell$.  

Suppose $A$ and $B$ are abelian varieties, each containing an embedding of a group scheme $G$ of order $\ell$.  Let $C$ be the quotient of $A \times B$ by a diagonal copy of $G \subset A \oplus A'$.  Then there are exact sequences 
\[0 \to A \to C \to B/G \to 0,\]     
\[0 \to B \to C \to A/G \to 0.\]
We therefore have the relation
\[[A] - [A_G] = [B] - [B_G]\]
in $K_0(\A)$, where $A_G$ (resp.\ $B_G$) is {\it any} quotient of $A$ (resp.\ $B$) by a group isomorphic to $G$.  So to prove the theorem, it suffices to find, for every prime $\ell$ and for every group scheme $G$ of order $\ell$, a single abelian variety $A$ and an endomorphism $f \in \End(A)$ such that $\ker f \simeq G$.  In fact, we will show that we can take $A$ to be an elliptic curve $E$.     

Write $p = \char \, k$.  If $\ell \neq p$, then $G \simeq \Z/\ell\Z$, and we may take $E$ such that $\End(E)$ contains $\Z[\sqrt{-\ell}]$.  Such an elliptic curve exists over any algebraically closed field $k$, by the theory of complex multiplication.  If $\ell = p$, then there are three group schemes $G$ to consider, but for all three we will take $E$ with $j$-invariant lying in $\F_p$.  In this case, the Frobenius morphism 
\[F \colon E \to E^{(p)} \simeq E\] 
is an endomorphism.  If $E$ is supersingular, then $\ker F \simeq \alpha_p$, while if $E$ is ordinary, then $\ker F \simeq \mu_p$ and $\ker \hat F \simeq \Z/p\Z$.  The number of supersingular elliptic curves over $\F_p$ is related to a certain class number by a result of Deuring \cite{deuring}, and is always non-zero  \cite[Thm.\ 14.18]{cox}.  On the other hand, the number of supersingular $j$-invariants over $\bar \F_p$ is less than $p$, so there are always $j$-invariants of both types in $\F_p$.  This concludes the proof.   
\end{proof}

\begin{corollary}\label{invariant}
Any additive function $\Ob(\A) \to G$ to an abelian group $G$ is an isogeny invariant.   
\end{corollary}
Corollary \ref{invariant} can be used to show that certain functions are {\it not} additive:
\begin{example}
If $k = \bar\Q$, then the stable Faltings height is additive under taking products of abelian varieties.  If it were additive under short exact sequences, then it would be an isogeny invariant. But it is easy to see from Faltings' isogeny formula \cite[Lem.\ 5]{faltings} that the height sometimes changes under isogeny.
\end{example}

\section{A useful criterion}

In this section, we let $k$ be any algebraically closed field.  Let $A$ be an abelian variety over $k$ of dimension $g$, not necessarily simple.  

\begin{lemma}\label{key}
Suppose $\pi_1: A \to A_1$ and $\pi_2 : A \to A_2$ are isogenies such that $\ker \pi_1 \cap \ker \pi_2 = 0$.  Then
\[[A] = [A_1] + [A_2] - [A_3] \in K_0(A),\]
where $A_3$ is the quotient $A/(\ker \pi_1 + \ker \pi_2)$.  
\end{lemma}

\begin{proof}
For $i = 1,2$, let $\tilde \pi_i : A_i \to A_3$ be the natural projection maps, so that 
\[\ker \tilde \pi_1 = \pi_1(\ker \pi_2) \hspace{5mm} \mbox{and} \hspace{5mm} \ker \tilde \pi_2 = \pi_2(\ker \pi_1).\]    
Then there is a short exact sequence
\[0 \to A \stackrel{\pi_1 \times \pi_2}{\longrightarrow} A_1 \times A_2 \stackrel{\tilde\pi_1 - \tilde \pi_2}{\longrightarrow} A_3 \to 0.\]
To see this, we need to check that the kernel of $\phi:= \tilde\pi_1 - \tilde \pi_2$ is contained in the image of $\pi_1 \times \pi_2$.  For concreteness, we argue on the level of points, leaving to the reader the exercise of making this argument categorical.\footnote{It is not enough to argue on the level of points if $\char\, k > 0$ and $\ker \pi_1 + \ker \pi_2$ has order divisible by $\char \, k$, but we will not actually use this case of the theorem.}  

So suppose $(P,Q) \in A_1 \times A_2$ is in $\ker \phi$.  Pick $\bar P \in A$ such that $\pi_1(\bar P) = P$.  Then it suffices to show that $Q = \pi_2(\bar P + R)$ for some $R \in \ker \pi_1$, because then \[(P,Q) = (\pi_1(\bar P + R), \pi_2(\bar P + R)).\]  
Now we compute
\[\tilde \pi_2 (\pi_2(\bar P) - Q) = \tilde\pi_1(\pi_1(\bar P)) - \tilde \pi_2 (Q) = \tilde \pi_1(P) - \tilde \pi_2(Q) = \phi(P,Q) = 0,\]
showing that $\pi_2(\bar P) - Q$ is contained in $\pi_2(\ker \pi_1) = \ker \tilde \pi_2$, as desired.  
\end{proof}

\begin{theorem}\label{samedeg}
If $A \to A_1$ and $A \to A_2$ are isogenies of the same degree $n$, and if $n$ is invertible in $k$, then $[A_1] = [A_2]$ in $K_0(A)$.  
\end{theorem}

\begin{proof}
First consider the case where the isogenies have prime degree $\ell$.  We may assume then that $\ker \pi_1 \neq \ker \pi_2$.  By Lemma \ref{key} we have
\[[A] + [A_3] = [A_1] + [A_2],\]
where $A_3 = A/(\ker \pi_1 + \ker \pi_2)$.  But the same argument works for any two distinct order $\ell$ subgroup schemes of $\ker \pi_1 + \ker \pi_2 \subset A[\ell]$.  Since $\ell$ is invertible in $k$, we can find a third such subgroup $C \subset \ker \pi_1 + \ker \pi_2$, and we have: 
\[[A_1] + [A/C] = [A] + [A_3] = [A_2] + [A/C],\]
and hence $[A_1] = [A_2]$.

The general case proceeds by induction on the degree.  If $A \to A_1$ and $A \to A_2$ have degree $n$, we can find a subgroup $C \subset A[n]$ of order $n$ which intersects non-trivially with $\ker \pi_1$ and $\ker \pi_2$.  Indeed, if $n = \ell^a$ is a prime power then one can take $C$ to contain any two points $P_1$ and $P_2$ of order $\ell$ in $\ker \pi_1$ and $\ker \pi_2$, respectively.  If $n = \prod \ell_i^{a_i}$ is not a prime power, then one can take $C$ to be generated by points of order $\ell_1$ and $\ell_2$ in $\ker \pi_1$ and $\ker \pi_2$, respectively.  Then the isogenies $A \to A_1$ and $A \to A/C$ factor through isogenies $A' \to A_1$ and $A' \to A/C$ of degree $n/\ell$ for some prime $\ell$.  By induction, $[A_1] = [A/C]$ and similarly $[A/C] = [A_2]$, so $[A_1] = [A_2]$.    
\end{proof}


\section{Simple isogeny classes}
In this section we let $\tilde A$ be an isogeny class of simple abelian varieties of dimension $g$.  Choose any $A \in \tilde A$ and set $D  = \End(A) \otimes_\Z \Q$.  Then $D$ is a division algebra whose isomorphism class depends only on $\tilde A$.  The degree map $\End(A) \to \Z$ extends to a multiplicative map $\deg: D^\times \to \Q_{+}$. 

\begin{remark}{
There may be elements of $\deg(D^\times) \cap \Z$ which are not of the form $\deg(\alpha)$ for some $\alpha \in \End(A)$.}  
\end{remark}    

\begin{lemma}\label{enddeg}
If $f: A_1 \to A_2$ is an isogeny in $\tilde A$ of degree $n$, such that $n$ is invertible in $k$ and $n = \deg(\beta)$ for some $\beta \in D^\times$, then $[A_1] = [A_2]$ in $K_0(A)$.  
\end{lemma}

\begin{proof}
Write $\beta  = ab^{-1}$ with $a,b \in \End(A_1) \cap \End(A_2)$, and with $\deg(a)$ invertible in $k$.  Here we are thinking of $\End(A_1)$ and $\End(A_2)$ as abstract rings embedded in $D$.  Then the composition
\[A_1 \stackrel{f}{\longrightarrow} A_2 \stackrel{b}{\longrightarrow} A_2\]
has the same degree as $a: A_1 \to A_1$.  By Theorem \ref{samedeg}, $[A_1] = [A_2]$.  \end{proof}

To compute $G(A) = \ker(K_0( A) \to \Z)$, it is convenient to choose $A \in \tilde A$ which is principally polarizable, so that $A \simeq \hat A$.  We write $\deg(D)$ for the submonoid $\deg(D^\times) \cap \Z$ of the monoid $\ZZ$ of positive integers under multiplication.  Note that $\ZZ\big/\deg(D)$ is a group.  

\begin{lemma}\label{dieu}
Suppose $f : A^n \to A^n$ is an isogeny.  Then $\deg(f) \in \deg(D)$.  
\end{lemma}

\begin{proof}
The isogeny $f$ can be thought of as an element $M \in \GL_n(D)$.  If $D$ is commutative then one has the formula
\[\deg(f) = \deg(\det M).\]
In the general case, there is no well behaved determinant map $\GL_n(D) \to D^\times$.  Instead, we have \[\deg(f) = (\Nm_{F/\Q}\circ \Nrd_n(M))^{2g/de},\] 
where $F$ is the center of $D$, $d = [F: \Q]$, $e^2 = [D: F]$, and $\Nrd_n: \GL_n(D) \to F^\times$ is the reduced norm; see \cite[\S 19]{mumford}.  On the other hand, for $g \in D$, we have
\[\deg(g) = (\Nm_{F/\Q} \circ \Nrd(g))^{2g/de},\]
where $\Nrd : D^\times \to F^\times$ is the reduced norm. 
The lemma then follows from Dieudonn\'e's result \cite{dieudonne} that $\Nrd_n(\GL_n(D)) = \Nrd(D^\times)$.         
\end{proof}

\begin{lemma}
Let $B \in \CC_A$ and let $f: A^n \to B$ be an isogeny.  Then the class of $\deg(f)$  in $\ZZ\big/\deg(D)$ is independent of the choice of $f$.   
\end{lemma}

\begin{proof}
Let $\phi_L: B \to \hat B$ be a polarization on $B$.  Then
\[\deg (\phi_{f^*L}) = \deg(\hat f \phi_L f) = \deg(f)^2 \deg(\phi_L).\] 
The isogeny $\phi_{f^*L} : A^n \to \hat A^n \cong A^n$ has degree in $\deg(D)$ by Lemma \ref{dieu}.     
Hence 
\begin{equation}\label{isopol}
\deg(f)^2 \deg(\phi_L) = 1 \in \ZZ\big/\deg(D).
\end{equation}
If $g: A^n \to B$ is another isogeny, then the composite map
\[h: A^n \stackrel{f}{\longrightarrow} B \stackrel{\phi_L}{\longrightarrow} \hat B \stackrel{\hat g}{\longrightarrow} \hat A^n \stackrel{\psi}{\longrightarrow} A^n\]
also has degree in $\deg(D)$.  Here, $\psi$ is any principal polarization.    
So modulo $\deg(D)$, we have  
\[1 = \deg(f)\deg(\hat g)\deg(\phi_L) = \deg(f)\deg(g)\deg(g)^{-2} = \deg(f)/\deg(g),\]
and therefore $\deg(f) = \deg(g)$ in $\ZZ\big/\deg(D)$.  
\end{proof}

\begin{definition}
If $B \in \CC_A$, then the class of $\deg(f)$ in $\ZZ\big/\deg(D)$, for any isogeny $f: A^n \to B$, is denoted $\dist_A(B)$ and is called the {\it distance} of $B$ from $A$.  
\end{definition}

\begin{lemma}\label{dual}
If $B \in \CC_A$, then $\dist_A(\hat B) = \dist_A( B)^{-1}$. 
\end{lemma}
\begin{proof}
From the sequence
\[A^n \longrightarrow B \stackrel{\phi_L}{\longrightarrow} \hat B,\]
we obtain $\dist_A(\hat B) = \dist_A(B)\deg(\phi_L)$, which is equal to $\dist_A(B)^{-1}$, by (\ref{isopol}).
\end{proof}

\begin{lemma}
Let $B \in \CC_A$ and let $g: B \to A^n$ be an isogeny.  Then the class of $\deg(g)$ in $\ZZ\big/\deg(D)$ is independent of the choice of $g$ and is equal to $\dist_A(B)^{-1}$.     
\end{lemma}

\begin{proof}
We have $\deg(g) = \deg(\hat g)$, with $\hat g: \hat A^n \to \hat B$ the dual isogeny.  Since $\hat A^n \cong A^n$, the class of $\deg(g)$ in $\ZZ\big/\deg(D)$ is independent of $g$ and equal to $\dist_A(\hat B)$.  Now use the previous lemma. 
\end{proof}

\begin{definition}
If $B \in \CC_A$, then the class of $\deg(f)$ in $\ZZ\big/\deg(D)$, for any isogeny $f:B \to A^n$, is denoted $\deg_A(B)$.  
\end{definition}

\begin{remark}\label{context}
We have $\deg_A(B)= \dist_A(B)^{-1}$ in $\ZZ\big/\deg(D)$.  Context dictates which invariant is most convenient to use.   
\end{remark}

\begin{remark}\label{indeprmk}
The map $\deg_A : \Ob(\CC_A) \to \ZZ\big/\deg(D)$ depends on the choice of $A$.  But note that if $f \colon B_1 \to B_2$ is an isogeny in $\CC_A$, then  
\[\frac{\deg_A(B_1)}{\deg_A(B_2)}\]
is equal to the class of $\deg(f) \in \ZZ\big/\deg(D)$ and hence is independent of the choice of $A$.  
\end{remark}

\begin{corollary}\label{ends}
If $B \in \CC_A$ and $\alpha \in \End(B)$ is any isogeny, then $\deg(\alpha) \in \deg(D)$.  
\end{corollary}

\section{Determination of $G(A)$ in characteristic 0}
Now we connect the notion of degree with the Grothendieck group.

\begin{proposition}\label{additive}
The function $\deg_A: \Ob(\CC_A) \to \ZZ\big/\deg(D)$ is additive.   
\end{proposition}

\begin{proof}
Let 
\[0 \to B_1 \stackrel{j}{\longrightarrow} B  \stackrel{\pi}{\longrightarrow} B_2 \to 0\]
be an exact sequence in $\CC_A$.  Let $\phi_M: \hat B_1 \to B_1$ and $\phi_L: B \to \hat B$ be polarizations and define $h: B \to B_1 \times B_2$ to be the map $\phi_M \hat j \phi_L \times \pi$.  From the sequence of isogenies  
\[B \stackrel{h}{\longrightarrow} B_1 \times B_2 \longrightarrow A^{n_1} \times A^{n_2} \cong A^n,\]
we conclude 
\[\deg_A(B) = \deg_A(B_1)\deg_A(B_2)\deg(h) \in \ZZ\big/\deg(D).\]
So it suffices to show that $\deg(h)$ is in $\deg(D)$. For this, note that $\ker h$ is isomorphic to the kernel of the isogeny $\phi_M \hat j \phi_L j \in \End(B_1)$.  Then by Corollary \ref{ends}, $\deg(h) \in \deg(D)$.  
\end{proof}

We therefore have a homomorphism $\deg_A : K_0(A) \to \ZZ\big/\deg(D)$.  By Remark \ref{indeprmk}, the restriction of $\deg_A$ to the dimension 0 subgroup $G(A) \subset K_0(A)$ is independent of $A$, so we write 
\[\deg: G(A) \to \ZZ\big/\deg(D).\] 
This homomorphism is surjective, since we work over an algebraically closed field.  In fact:  
\begin{theorem}\label{main}
Suppose $\char\,  k = 0$.  Then the degree map $\deg: G(A) \to \ZZ\big/\deg(D)$ is an isomorphism.
\end{theorem}

\begin{proof}
We write $A_n$ for any $A_n \in \tilde A$ such that $\dist_A(A_n) = n$.  The class $[A_n] \in K_0(A)$ is independent of the choice of $A_n$ by Theorem \ref{samedeg}.  By Lemma \ref{key}, we have, for any $m,n \in \ZZ$:
\begin{equation}\label{decomp}
[A_{nm}] = [A_n] + [A_m] - [A],
\end{equation} because we can always represent $A_n$ and $A_m$ by cyclic quotients of $A$ with non-intersecting kernels $C_n$ and $C_m$ (since $\char \, k = 0$), and $A_{mn}$ by the quotient $A/(C_m + C_n)$.  

Note that $K_0(A)$ is generated by classes $[A']$ of {\it simple} $A' \Ob(\CC_A)$.  Thus, any $\beta \in G(A)$, can be written as  
\[\beta = \sum_{i = 1}^r \left([A_{n_i}] - [A_{m_i}]\right) = [A_{\prod_i^r n_i}] + (r-1)[A] - [A_{\prod_i^r m_i}] - (r-1)[A] = [A'] - [A''], \]
for certain $A', A'' \in \tilde A$ .  If $\beta$ is also in the kernel of the degree map $G(A) \to \ZZ\big/\deg(D)$, then 
\[\deg_A(A')/\deg_A(A'') \in \deg(D).\] Equivalently, there is an isogeny $f: A' \to A''$ of degree $\deg(\alpha)$ for some $\alpha \in D$.  By Lemma \ref{enddeg}, $[A'] = [A'']$ and $\beta = 0$, showing that $G(A) \to \ZZ
\big/\deg(D)$ is injective, and hence an isomorphism.     
\end{proof}

As a corollary, we obtain Theorem \ref{duality}:
\begin{corollary}
The additive function $B \mapsto \hat B$ induces the inversion homomorphism on $G(A)$.  
\end{corollary}

\begin{proof}
Since $\deg_A(\hat B) = \deg_A(B)^{-1}$, by Lemma \ref{dual} and Remark \ref{context}.  
\end{proof}

The following proposition gives a concrete description of $\ZZ\big/\deg(D)$.  
\begin{proposition}\label{concrete}
Let $A$ be a simple abelian variety of dimension $g$ and $D = \End( A)\otimes \Q$ its endomorphism algebra, with center $Z(D) = F$.  Write $e = [F: \Q]$ and $d^2 = [D: F]$.  Then
\[\ZZ\big/\deg(D) \simeq \QQ\big /\Nm_{F/\Q}(F^+)^{2g/de},\]
where $F^+$ is the set of non-zero elements of $F$ which are positive at all the ramified real places of $D$.    
\end{proposition}

\begin{proof}
We have seen already that $\deg : D^\times \to \Q_{+}$ is given by the map $(\Nm_{F/\Q} \circ \Nrd)^{2g/de}$.  But the Hasse-Schilling-Maass theorem \cite[Thm\ 33.15]{reiner} states that the reduced norm on $D$ surjects onto $F^+$.    
\end{proof}

Theorem \ref{intro} now follows from Theorem \ref{main} and Proposition \ref{concrete}

\section{The case $\dim A = 1$}

In this section we determine $G(E)$, for any elliptic curve $E$ over any algebraically closed field $k$, of any characteristic.  The characteristic 0 cases can be read off from Theorem \ref{intro}:
\begin{theorem}\label{char0}
Suppose $k = \bar k$ has characteristic $0$ and $E/k$ is an elliptic curve with endomorphism algebra $D$.  Then the degree map induces an isomorphism
\[G(E) \simeq 
\begin{cases}
\QQ/\qq & \mbox{ if $D = \Q$},\\
\QQ/\Nm_{K/\Q}(K^\times) & \mbox{ if $D = K$ is imaginary quadratic.}
\end{cases}  
\]
\end{theorem}

In the CM case, we can make the group structure of $G(E)$ more explicit:

\begin{proposition}\label{classgroup}
If $K$ is imaginary quadratic over $\Q$, then 
\[\QQ/\Nm_{K/\Q}(K^\times) \simeq C/C^2\oplus \bigoplus_{\mbox{$\ell$ inert}} \Z/2\Z,\]
where $C  = \Pic(\O_K)$ is the class group of $K$.  
\end{proposition}

\begin{proof}
If $\ell = \Nm(\alpha)$ for some $\alpha \in K^\times$, then $\ell$ is not inert in $K$ and $(\alpha) = \l \a \bar\a^{-1}$ for some ideal $\a$ of $\O_K$ and some prime $\mathfrak{l}$ above $\ell$.  It follows that $[\l]$ is a square in $C$.  We therefore get a well defined map
\[\QQ/\Nm_{K/\Q}(K^\times) \longrightarrow C/C^2\oplus \bigoplus_{\mbox{$\ell$ inert}} \Z/2\] by sending an inert prime $\ell$ to the generator of $\Z/2\Z$ in the $\ell$th slot, and sending a non-inert prime $\ell$ to the class of $[\l]$ in  $C/C^2$, where $\l$ is any prime above $\ell$.  This map is clearly surjective, and to prove injectivity, we need to show that if $[\l]$ is a square in $C$, then $\ell$ is a norm.  If $[\l]$ is a square, then $\l = (\alpha)\a^2 = (\beta)\a\bar\a^{-1}$ for some ideal $\a$, so we see that $\Nm(\beta) = \ell$. 
\end{proof}

Now suppose $k$ has characteristic $p > 0$.  There are three cases to consider, depending on the dimension of $D = \End(E) \otimes \Q$ over $\Q$.    

\begin{theorem}
If $E$ is supersingular, then $G(E) = 0$.   
\end{theorem}
\begin{proof}
The isogeny class of $E$ is the set of supersingular elliptic curves.  It is therefore enough to show that $[E']  =[E]$ in $K_0(E)$ for any other supersingular elliptic curve $E'$.  By \cite[Cor.\ 77]{kohel}, we may choose a prime number $\ell \neq p$ such that there exist $\ell$-isogenies $E \to E$ and $E \to E'$.  Then $[E'] = [E]$ in $K_0(E)$ by Theorem \ref{samedeg}. 	    
\end{proof}

\begin{theorem}\label{ordinary}
If $D$ is isomorphic to an imaginary quadratic field $K$, then the degree map induces an isomorphism $G(E) \simeq \QQ/\Nm_{K/\Q}(K^\times)$.  
\end{theorem}

\begin{proof}
We may assume that $\End(E)$ is isomorphic to the ring of integers $\O_K$.  Let $E'$ be an elliptic curve isogenous to $E$.  By a result of Deuring \cite[p.\ 263]{deuring}, the ring $\End(E')$ has index prime to $p$ in $\O_K$ and $p$ is split in $\O_K$.  Thus, by \cite[Prop.\ 40]{kani}, the rank two quadratic form $\deg \colon \Hom(E, E') \to \Z$, has discriminant prime to $p$.  It follows that there is an isogeny $E \to E'$ of degree prime to $p$.  Now we proceed exactly as in the proof of Theorem \ref{main}, but using the fact that $[E'] = [E_n]$ in $K_0(E)$, for some $n \in \Z$ prime to $p$.       
\end{proof}

\begin{theorem}\label{trans}
If $D = \Q$, then $G(E)$ is isomorphic to a subgroup of index $2$ in $\Z \bigoplus \QQ\big/\qq$.  In particular, $G(E)$ is not a torsion group.    
\end{theorem}

\begin{proof}
We define an additive map $\deg_{p,E}  \colon \Ob(\CC_E) \to \Z$ as follows.  For any isogeny $f \colon B \to B'$ in $\CC_E$, let $e(f)$ denote the number of Jordan-Holder factors of $\ker f$ isomorphic to $\Z/p\Z$, let $c(f)$ denote the number of factors isomorphic to $\mu_p$, and let $\deg_p(f) = e(f) - c(f)$.  For $B \in \Ob(\CC_E)$,  we define $\deg_{p,E}(B) = \deg_p(f)$, where $f$ is any isogeny $f \colon B \to E^n$.  To check that this is well-defined we use:  

\begin{lemma}
If $f \colon E^n \to E^n$ is an isogeny, then $\deg_p(f) = 0$.   
\end{lemma}    

\begin{proof}
The case $n = 1$ is clear since then $f \in \Z$ and $E[p] \simeq \Z/p\Z \oplus \mu_p$.  For $n > 1$, we may think of $f$ as a matrix $M \in \Mat_n(\Z)$.  We may assume $M$ is in Smith normal form, at the cost of choosing a new product decomposition for $E^n$.  Then $\ker f$ is the direct sum of kernels of endomorphisms of $E$ and the lemma follows from the case $n = 1$.   
\end{proof}
Now suppose $f' \colon B \to E^n$ is another isogeny.  If we let $g' \colon E^n \to B$ be such that $f'g' = [\deg f']_{E^n}$, then $\deg_p(f') + \deg_p(g') = 0$ and 
\[\deg_p(f) - \deg_p(f') = \deg_p(f) + \deg_p(g') = \deg_p(fg') = 0,\]
by the lemma.  So $\deg_{p,E}(B)$ is well-defined.  We also conclude that $\deg_p(f) = 0$ for any $f \in \End(B)$, just as in Corollary \ref{ends}.  

One now checks that $\deg_{p,E}$ is additive, exactly as in Proposition \ref{additive}.  We therefore get an induced map $K_0(E) \to \Z$ which depends on $E$.  But the restriction to $G(E)$ gives a canonical map $\deg_p \colon G(E) \to \Z$  sending $[E'] - [E'']$ to $\deg_p (f)$, where $f \colon E' \to E''$ is any isogeny.  

For any isogeny $f$, we write $\ldeg(f)$ for the prime-to-$p$ part of $\deg(f)$. Now consider the composite map
\[\tot \colon G(E) \stackrel{(\deg_p, \deg)}{\longrightarrow} \Z \oplus \QQ\big /\Q_+^2 \longrightarrow \Z \oplus H_p\]
where $H_p \simeq \bigoplus_{\ell \neq p} \Z/2\Z$ is the prime-to-$p$ part of $\QQ\big /\Q_+^2$, and the map $\QQ\big /\Q_+^2 \to H_p$ is the canonical surjection.  Then 
\[\tot([E_1] - [E_2]) = (\deg_p(f), \deg_\ell(f))\] 
for any isogeny $f \colon E_1 \to E_2$.  It follows that $\tot$ is surjective, and we will show that it is injective as well.    

For every $(a,n) \in \Z \oplus \Z_+$ with $n$ prime to $p$, we let $E_{a,n}$ be any elliptic curve admitting an isogeny $f \colon E \to E_{a,n}$ such that $\deg_p(f) = a$ and $\ldeg(f) = n$.  Then the class of $[E_{a,n}]$ in $K_0(E)$ is independent of the choice of $E_{a,n}$.  Indeed, the isogeny $f \colon E \to E_{a,n}$ can be factored as 
\[E  \longrightarrow E_a \stackrel{g}{\longrightarrow} E_{a,n}, \]
where $E_a$ is the unique \'etale quotient of $E$ of degree $p^a$, and $g$ is an $n$-isogeny.  Thus the class of $[E_{a,n}]$ is uniquely determined by Theorem \ref{samedeg}.  

Using Lemma \ref{key}, we obtain the following relations in $K_0(E)$, for any integers $a$ and $b$, and any positive integers $n$ and $m$ coprime to $p$:
\begin{align*}
[E_{a,n}] &= [E_{a,1}] + [E_{0,n}] - [E]\\
[E_{a+b,1}] &= [E_{a,1}] + [E_{b,1}] - [E]\\
[E_{0,nm}] &= [E_{0,n}] + [E_{0,m}] - [E].
\end{align*}
It follows that any $\beta \in G(E)$ can be written as 
\[\beta = [E_{a,1}] - [E_{b,1}] + [E_{0,n}] - [E_{0,m}]\] for integers $a$ and $b$, and positive integers $n$ and $m$.  If $\tot(\beta) = 0$, then we must have $a = b$ and $n = m d^2$ for some rational number $d$.  By Lemma \ref{enddeg}, we must have $[E_{0,n}] = [E_{0,m}]$, and so $\beta = 0$.  Thus, $\tot$ is an injection, and hence an isomorphism.  Since $\Z \oplus H_p$ embeds in $\Z \oplus \QQ\big /\Q_+^2$ as an index 2 subgroup, Theorem \ref{trans} is proved.  
\end{proof}

\begin{remark}
An example of a non-torsion class in $G(E)$ is $[E] - [E^{(p)}]$, where $E^{(p)}$ is the Frobenius-transform of $E$, i.e.\ the elliptic curve with $j$-invariant $j(E)^p$.  
\end{remark}

\subsection{Acknowledgments} 
The author thanks Julian Rosen for suggesting the proof of Theorem \ref{entirecat} and for other conversations on this topic.  He also thanks Taylor Dupuy and Yuri Zarhin for their helpful feedback.      



\begin{thebibliography}{12}

\bibitem[C]{cox}
D.\ A.\ Cox, Primes of the form $x^2 + ny^2$, John Wiley, 1989.

\bibitem[De]{deuring}
 M.\ Deuring, Die Typen der Multiplikatorenringe elliptischer Funktionenkorper.
{\it Abh.\ Math.\ Sem.\ Hamburg} {\bf 14} (1941), 197-272.

\bibitem[Di]{dieudonne} 
J.\ Dieudonn\'e, Les d\'eterminants sur un corps non commutatif, {\it Bull.\ Soc.\ Math.\ France}  {\bf 71} 
(1943). 27-45. 

\bibitem[F]{faltings} 
G.\ Faltings, Endlichkeitss\" aatze f\"ur abelsche Variet\"aten\"uber Zahlk\"orpern, {\it Invent.\ Math.\ }  {\bf 73} (1983), 349-366.

\bibitem[$\mbox{JKP}^+$]{jordan}
B.\ Jordan, A.\ Keeton, B.\ Poonen, E.\ Rains, N.\ Shepherd-Barron, and J.\ Tate, Abelian varieties isogenous to a power of an elliptic curve", preprint, arXiv:1602.06237 (2016).

\bibitem[Ka]{kani}
E.\ Kani, Products of CM Elliptic Curves, {\it Collectanea Math.} {\bf 62}, (2011) 297-339.

\bibitem[Ko]{kohel}
D.\ Kohel, {\it Endomorphism rings of elliptic curves over finite fields}, PhD thesis, University
of California at Berkeley, 1996.

\bibitem[Mi]{milne} 
J.\ S.\ Milne, Extensions of abelian varieties defined over a finite field. {\it Invent.\ Math.\ } {\bf 5} (1968),
63-84.

\bibitem[Mu]{mumford} 
D.\ Mumford, {\it Abelian varieties}, Oxford Univ. Press, 1970.

\bibitem[R]{reiner} 
 I.\ Reiner, {\it Maximal Orders}, London Math.\ Soc.\ Monographs {\bf 5}, Academic Press, London,
1975.

\end{thebibliography}
\end{document}